\documentclass{amsart}
\usepackage{amsmath, amscd, amssymb, amsthm}
\usepackage{bbm}
\usepackage{latexsym}
\usepackage{amsfonts}
\usepackage{graphicx}
\usepackage[all,cmtip]{xy}
\usepackage[colorlinks,linkcolor=blue,breaklinks=blue,urlcolor=blue,citecolor=blue,anchorcolor=blue,pagebackref]{hyperref}%
\usepackage{geometry}
\newtheorem{theorem}{Theorem}
\newtheorem{lemma}{Lemma}

\newtheorem*{example}{Example}

\newtheorem{conjecture}{Conjecture}

\renewcommand*\backref[1]{}
\renewcommand*\backrefalt[4]{ \ifcase #1 \or (cited on page #2) \else (cited on pages #2) \fi}

\newcommand{\be}{\begin{equation}}
\newcommand{\ee}{\end{equation}}
\newcommand{\bea}{\begin{eqnarray}}
\newcommand{\eea}{\end{eqnarray}}

\newcommand{\vs}{\vspace{0.5cm}}

\def\XXint#1#2#3{{\setbox0=\hbox{$#1{#2#3}{\int}$ }
\vcenter{\hbox{$#2#3$ }}\kern-.6\wd0}}

\begin{document}

\title[Solvmanifolds with complex commutator]{On solvmanifolds with complex commutator and constant holomorphic sectional curvature}

\author{Xin Huang}
\address{Xin Huang. School of Mathematical Sciences, Chongqing Normal University, Chongqing 401331, China}
\email{{1772503556@qq.com}}\thanks{The corresponding author Zheng is partially supported by NSFC grants 12141101 and 12471039,  by Chongqing Normal University grant 24XLB026, and is supported by the 111 Project D21024.}

\author{Fangyang Zheng}
\address{Fangyang Zheng. School of Mathematical Sciences, Chongqing Normal University, Chongqing 401331, China}
\email{20190045@cqnu.edu.cn; franciszheng@yahoo.com} \thanks{}


\begin{abstract}
An old open question in non-K\"ahler geometry predicts that any compact Hermitian manifold with constant holomorphic sectional curvature must be K\"ahler or Chern flat.  The conjecture is known to be true in dimension $2$ due to the work by Balas-Gauduchon and Apostolov-Davidov-Muskarov in the 1980s and 1990s,  but is still open in dimensions $3$ or higher, except in several special cases. The difficulty in this quest for `Hermitian space forms' is largely due to the algebraic complicity or lack of symmetry for the curvature tensor of a general Hermitian metric. In this article, we confirm the conjecture for all solvmanifolds with complex commutator, extending earlier result on nilmanifolds by Li and the second named author. 
\end{abstract}

\maketitle

\tableofcontents

\section{Introduction}\label{intro}

The simplest kind of spaces in differential geometry are those with constant curvature. In the Riemannian case, complete Riemannian manifolds with constant {\em sectional curvature} are called (real) {\em space forms,} and it is well-known that their universal covers are either the sphere $S^n$, or the Euclidean space ${\mathbb R}^n$, or the hyperbolic space ${\mathbb H}^n$, equipped with the (scaling of the) standard metrics. In the complex case, a K\"ahler metric can no longer have constant sectional curvature unless it is flat, so instead {\em complex space forms} mean complete K\"ahler manifolds with constant {\em holomorphic sectional curvature}. Their universal covers are known to be either the complex projective space  ${\mathbb C}{\mathbb P}^n$, or the complex Euclidean space ${\mathbb C}^n$, or the complex hyperbolic space ${\mathbb C}{\mathbb H}^n$, equipped with the (scaling of the) standard metrics. It is a natural question to wonder what happens if we drop the K\"ahlerness assumption, namely, what kind of complete Hermitian manifolds can have constant holomorphic sectional curvature? In this aspect, the following is a long-standing conjecture in complex geometry:

\begin{conjecture} \label{conj1}
If the Chern connection of a compact Hermitian manifold has constant holomorphic sectional curvature, then it must be either K\"ahler (hence a complex space form) or Chern flat.
\end{conjecture}

Let us give a few quick remarks. First of all, if one replaces the Chern connection by Levi-Civita, Bismut, or other Gauduchon connections in the conjecture, then there are corresponding conjectures/questions. See for instance \cite{CN, ChenZ, ChenZ1} and the references therein for more discussions. Secondly, the compactness requirement in the assumption is necessary, without it there are counterexamples. Thirdly, by the classic theorem of Boothby \cite{Boothby}, compact Chern flat manifolds are exactly compact quotients of complex Lie groups equipped with left-invariant metrics. They are in general non-K\"ahler except when the Lie group is abelian (so the manifold is a finite undercover of a flat complex torus), and there are plenty of such examples in complex dimensions $3$ or higher.  

Conjecture \ref{conj1} is known to be true in dimension $2$. When the constant holomorphic sectional curvature is zero or negative, it was proved by Balas and Gauduchon  in 1985 (\cite{Balas, BG}). When the constant is positive, it was due to Apostolov, Davidov, and Muskarov \cite{ADM} in 1996. For Levi-Civita connection, the $n=2$ case was due to Sato and Sekigawa \cite{SatoSeki} in 1990 for the constant zero or negative case, and the constant positive case was again due to \cite{ADM}. The $n=2$ case for Bismut connection was due to Chen and the second named author in \cite{ChenZ}, while for Gauduchon connections it was due to Chen and Nie in \cite{CN}.  

In dimensions $3$ or higher, Conjecture \ref{conj1} is still largely open, except in some special situations when the metric satisfies various additional assumptions. In \cite{Tang}, Tang verified the conjecture under the additional assumption that the metric is {\em Chern K\"ahler-like,} a notion introduced by Yang and the second named author \cite{YangZ}, meaning that the curvature tensor of the Chern connection obeys all K\"ahler symmetries. Chen, Chen and Nie \cite{CCN} showed that if the metric is locally conformally K\"ahler and if the constant is zero or negative, then Conjecture \ref{conj1} holds. In \cite{ZhouZ}, Zhou and the second named author proved that any compact Hermitian threefold with vanishing {\em real bisectional curvature} must be Chern flat. Real bisectional curvature is a curvature notion introduced by X. Yang and the second named author in \cite{XYangZ}. It is equivalent to holomorphic sectional curvature $H$ in strength when the metric is K\"ahler, but is slightly stronger than $H$ when the metric is non-K\"ahler. In \cite{RZ}, Rao and the second named author showed that the conjecture holds if the metric is {\em Bismut K\"ahler-like,} meaning that the curvature of the Bismut connection obeys all K\"ahler symmetries. 

Recall that a {\em Lie-Hermitian manifold} is a  compact quotient $M=G/\Gamma$ of a Lie group $G$ by a discrete subgroup $\Gamma \subseteq G$, with the (lift of the) complex structure and metric on $G$ being left-invariant. $(M^n,g)$ is called a {\em complex nilmanifold} or {\em complex solvmanifold} if $G$ is nilpotent or solvable, respectively. 

Lie-Hermitian manifolds form a large and interesting class of special Hermitian manifolds with ample algebraic complexity. They are often used as testing ground in non-K\"ahler geometry. In the nilpotent case,  Li and the second named author showed in \cite{LZ} that Conjecture \ref{conj1} holds for all complex nilmanifolds. 

Let us denote by ${\mathfrak g}$ the Lie algebra of $G$. Then left-invariant metrics on $G$ correspond to metrics (namely, inner products) on the vector space ${\mathfrak g}$, and left-invariant complex structures on $G$ correspond to complex structures on ${\mathfrak g}$, namely, linear transformations $J:{\mathfrak g} \rightarrow {\mathfrak g}$ satisfying $J^2=-I$ and the integrability condition  
\begin{equation} \label{integrability}
[x,y] - [Jx,Jy] + J[Jx,y] + J[x,Jy] =0, \ \ \ \ \ \forall \ x,y \in {\mathfrak g}. 
\end{equation}
We will call a Lie algebra ${\mathfrak g}$ equipped with a complex structure $J$ and a compatible metric $g=\langle, \rangle$  a {\em Hermitian Lie algebra,} denoted by $({\mathfrak g}, J, g)$. Here and from now on, we will use the same letter to denote the metric/complex structure on $G$ and ${\mathfrak g}$.

The main purpose of this article is to verify the conjecture for a special type of Lie-Hermitian manifold: when the Lie algebra ${\mathfrak g}$ of $G$ is solvable and when $J{\mathfrak g}'={\mathfrak g}'$, where ${\mathfrak g}'=[{\mathfrak g},{\mathfrak g}]$ denotes the commutator of ${\mathfrak g}$. For simplicity, sub Lie algebras will be called {\em solvable Lie algebras with complex commutators} from now on.

\begin{theorem} \label{thm}
Let $(M^n,g)$ be a Lie-Hermitian manifold with universal cover $(G,J,g)$. Denote by ${\mathfrak g}$ the Lie algebra of $G$ and ${\mathfrak g}'=[{\mathfrak g},{\mathfrak g}]$ its commutator. Assume that ${\mathfrak g}$ is solvable and $J{\mathfrak g}'={\mathfrak g}'$. If the Chern connection of $g$ has constant holomorphic sectional curvature, then it must be Chern flat. 
\end{theorem}

Note that the assumption of the theorem does not put any restriction on the solvable Lie algebra  ${\mathfrak g}$ other than its commutator ${\mathfrak g}'$ is even-dimensional, but it does put a strong restriction on the complex structure $J$, which must preserve ${\mathfrak g}'$. Nonetheless, there are still plenty of such Hermitian Lie algebras, and we will see some examples in \S 4. 

Finally let us remark on the compactness assumption in Conjecture \ref{conj1}. As mentioned before, the conjecture fails without the compactness assumption, and explicit counterexamples were given in \cite{CCN}. For Lie-Hermitian manifolds, we believe that there might be counterexamples as well for non-unimodular Lie algebras. However, in the special case of solvable Lie algebra with complex commutator, one does not need to assume the Lie algebra to be unimodular, in other words, Theorem \ref{thm} can be rephrased as the following:

{\em If ${\mathfrak g}$ is a solvable Lie algebra with a Hermitian structure so that the commutator is complex and the Chern holomorphic sectional curvature is constant, then it is Chern flat. }

\vspace{0.3cm}

\section{Hermitian Lie algebras}

Let us begin by recalling the definition of holomorphic sectional curvature. Given a Hermitian manifold $(M^n,g)$, let us denote by $\nabla$ the Chern connection, and by $T$, $R$ its torsion and curvature tensor, defined by
$$ T(x,y)=\nabla_xy-\nabla_yx-[x,y], \ \ \ \ \ R(x,y,z,w)=\langle \nabla_x\nabla_yz - \nabla_y\nabla_x z -\nabla_{[x,y]}z , w\rangle ,\ \ \ $$
respectively, where $x,y,z,w$ are tangent vectors on $M$, and we wrote $g=\langle , \rangle$.  Clearly, $R$ is skew-symmetric with respect to its first two positions, and since $\nabla$ is a metric connection (namely, since $\nabla g=0$), $R$ is also skew-symmetric with respect to its last two positions. Therefore one can define the {\em sectional curvature} of $\nabla$ by
$$ K(\pi ) = \frac{ R(x,y,y,x)}{|x|^2|y|^2 - \langle x,y\rangle ^2}, $$
for any $2$-dimensional subspace $\pi$ in the tangent space, with $\{ x,y\}$ being any basis of $\pi$. Clearly this value is independent of the choice of the basis. When $\pi$ is $J$-invariant, or equivalently, when there is a basis of $\pi$ in the form $\{ x, Jx\}$, then we get the {\em holomorphic sectional curvature}
$$ H = \frac{R(x,Jx,Jx,x)}{|x|^4}. $$
If we write $X=x-\sqrt{-1}Jx$, then the above value is equal to $H(X) = R(X, \overline{X}, X, \overline{X})/|X|^4$, which is perhaps the expression that people are more familiar with. Note that sectional curvature and holomorphic sectional curvature can be defined for any {\em metric connection} of $(M^n,g)$, not just the Chern connection, even though in this article we will only be interested in the Chern connection $\nabla$. Let $\{ e_1, \ldots , e_n\}$ be a local frame of type $(1,0)$ complex tangent vectors, write $g_{i\bar{j}}=\langle e_i, \overline{e}_j\rangle$, and let $R_{i\bar{j}k\bar{\ell}} $ be the short hand notation for $R(e_i, \overline{e}_j, e_k, \overline{e}_{\ell})$. As in \cite{LZ}, let us introduce the {\em symmetrization} of a $(4,0)$-tensor:
$$ \widehat{R}_{i\bar{j}k\bar{\ell}} = \frac{1}{4} \big( R_{i\bar{j}k\bar{\ell}} + R_{k\bar{j}i\bar{\ell}} + R_{i\bar{\ell}k\bar{j}} + R_{k\bar{\ell}i\bar{j}} \big) , $$
for any $1\leq i,j,k,\ell \leq n$. Then it is easy to see that

\begin{lemma} \label{lemma1}
Let $(M^n,g)$ be a Hermitian manifold. Then we have
$$ H=c \ \  \Longleftrightarrow \ \ \widehat{R}_{i\bar{j}k\bar{\ell}} = \frac{c}{2}\big( g_{i\bar{j}} g_{k\bar{\ell}} + g_{i\bar{\ell }} g_{k\bar{j}}  \big) , \ \ \forall \ 1\leq i,j,k,\ell \leq n. $$
\end{lemma}

From the formula, one sees clearly that the values of holomorphic sectional curvature $H$ could only control the symmetrization part of $R$, instead of the entire $R$. This is the main source of difficulty when dealing with Conjecture \ref{conj1}. 

Next let us consider Lie-complex manifolds, namely, compact quotients of Lie groups by their discrete  subgroups, where the complex structures are left-invariant. In the past a few decades, the Hermitian geometry of Lie-complex manifolds have been extensively studied from various aspects by many people, including A. Gray, S. Salamon, L. Ugarte, A. Fino, L. Vezzoni, F. Podest\`a, D. Angella, A. Andrada, and others. There is a large amount of literature on this topic, and here we will just mention a small sample: \cite{AU}, \cite{CFGU},  \cite{FP3}, \cite{FinoTomassini09},  \cite{GiustiPodesta}, \cite{Salamon}, \cite{Ugarte}, \cite{WYZ}, \cite{ZZ-Crelle}, \cite{ZZ-JGP}. For more general discussions on non-K\"ahler Hermitian geometry with a broader view, see for example \cite{AI}, \cite{AT}, \cite{Fu}, \cite{STW}, \cite{Tosatti} and the references therein.

Let $M=G/\Gamma$ be a compact quotient of a Lie group $G$ by a discrete subgroup $\Gamma \subseteq G$, equipped with (the descend of) a left-invariant complex structure $J$ on $G$.  Let ${\mathfrak g}$ be the Lie algebra of $G$, then $J$ corresponds to a complex structure (which for convenience we will still denote by $J$) on  ${\mathfrak g}$. Similarly, any left-invariant metric $g$ on $G$ compatible with $J$ will correspond to an inner product $g=\langle , \rangle$ on ${\mathfrak g}$ such that $\langle Jx,Jy\rangle = \langle x,y\rangle$ for any $x,y\in {\mathfrak g }$. Here again for convenience we use the same letter to denote the corresponding metric on the Lie algebra. 

Following \cite{VYZ},  let ${\mathfrak g}^{\mathbb C}$ be the complexification of ${\mathfrak g}$, and write ${\mathfrak g}^{1,0}= \{ x-\sqrt{-1}Jx \mid x \in {\mathfrak g}\} \subseteq {\mathfrak g}^{\mathbb C}$. The integrability condition (\ref{integrability}) means that ${\mathfrak g}^{1,0}$ is a complex Lie subalgebra of ${\mathfrak g}^{\mathbb C}$. As in \cite{VYZ} or \cite{GuoZ2}, let us extend $g=\langle , \rangle $ bi-linearly over ${\mathbb C}$, and let $e=\{ e_1, \ldots , e_n\}$ be a basis of ${\mathfrak g}^{1,0}$, which will be called a {\em frame} of $({\mathfrak g},J)$ or $({\mathfrak g}, J,g)$.  Denote by $\varphi$ the coframe dual to $e$, namely, a basis of the dual vector space $({\mathfrak g}^{1,0})^{\ast}$ such that $\varphi_i(e_j)=\delta_{ij}$, $\forall$ $1\leq i,j\leq n$. We will write
\begin{equation} \label{CandD}
C^j_{ik} = \varphi_j( [e_i,e_k] ), \ \ \ \ \ \  D^j_{ik} = \overline{\varphi}_i( [\overline{e}_j, e_k] )
\end{equation}
for the structure constants, which is  equivalent to
\begin{equation} \label{CandD2}
[e_i,e_j] = \sum_k C^k_{ij}e_k, \ \ \ \ \ [e_i, \overline{e}_j] = \sum_k \big( \overline{D^i_{kj}} e_k - D^j_{ki} \overline{e}_k \big) .
\end{equation}
In dual terms, the above is also equivalent to the (first) structure equation:
\begin{equation} \label{structure}
d\varphi_i = -\frac{1}{2} \sum_{j,k} C^i_{jk} \,\varphi_j\wedge \varphi_k - \sum_{j,k} \overline{D^j_{ik}} \,\varphi_j \wedge \overline{\varphi}_k, \ \ \ \ \ \ \forall \  1\leq i\leq n.
\end{equation}
Differentiate the above, we get the  first Bianchi identity, which is equivalent to the Jacobi identity in this case:
\begin{equation} \label{Bianchi}
\left\{  \begin{split}  \sum_r \big( C^r_{ij}C^{\ell}_{rk} + C^r_{jk}C^{\ell}_{ri} + C^r_{ki}C^{\ell}_{rj} \big) \ = \ 0,  \hspace{3.2cm}\\
 \sum_r \big( C^r_{ik}D^{\ell}_{jr} + D^r_{ji}D^{\ell}_{rk} - D^r_{jk}D^{\ell}_{ri} \big) \ = \ 0, \hspace{3cm} \\
 \sum_r \big( C^r_{ik}\overline{D^r_{j \ell}}  - C^j_{rk}\overline{D^i_{r \ell}} + C^j_{ri}\overline{D^k_{r \ell}} -  D^{\ell}_{ri}\overline{D^k_{j r}} +  D^{\ell}_{rk}\overline{D^i_{jr}}  \big) \ = \ 0,  \end{split} \right.  
\end{equation}
for any $1\leq i,j,k,\ell\leq n$. When $G$ has a compact quotient, it (or equivalently, its Lie algebra ${\mathfrak g}$) is necessarily {\em unimodular,} that is,  $\mbox{tr}(ad_x)=0$ for any $x\in {\mathfrak g}$. In terms of structure constants, 
\begin{equation*} \label{unimodular}
{\mathfrak g} \ \, \mbox{is unimodular}  \ \ \Longleftrightarrow  \ \ \sum_{s=1}^n \big( C^s_{si} + D^s_{si}\big) =0 , \, \ \forall \ i.
\end{equation*}
Denote by  $\nabla$ the Chern connection of $g$, and by $T$, $R$ the torsion and curvature tensor of $\nabla$. It is well-known that under any frame $e$, one always has $T( e_i, \overline{e}_j)=0$, while any $T(e_i,e_k)$ is always of type $(1,0)$. Let us write $ T(e_i,e_k) = \sum_k T^j_{ik}e_j$. Also, the only possibly non-zero components of $R$ are $R(e_i, \overline{e}_j , e_k, \overline{e}_{\ell}) =R_{i\bar{j}k\bar{\ell}}$. 

Note that so far we did not assume the frame $e$ to be unitary. When $e$ is unitary, namely, when $\langle e_i, \overline{e}_j\rangle = \delta_{ij}$ for any $1\leq i,j\leq n$, then the formula for the Chern torsion components $T^j_{ik}$ and Chern curvature components $R_{i\bar{j}k\bar{\ell}}$  under $e$ are particularly simple (\cite{GuoZ2, LZ}):

\begin{lemma}  \label{lemma2}
Given a Hermitian Lie algebra $({\mathfrak g}, J,g)$, let $e$ be a unitary frame, and structure constants $C$ and $D$ be given by (\ref{CandD}). Then the Chern torsion components $T^j_{ik}$ and Chern curvature components  $R_{i\bar{j}k\bar{\ell}}$  under $e$ are given by
\begin{eqnarray} 
  T^j_{ik} & = &  - C^j_{ik} -   D^{j}_{i k} + D^{j}_{k i} , \label{torsion} \\
 R_{i\bar{j}k\bar{\ell}} & = &  \sum_{s=1}^n \big( D^s_{ki}\overline{D^s_{\ell j}} - D^{\ell}_{si}\overline{D^k_{s j}} - D^j_{si}\overline{D^k_{ \ell s}} - \overline{D^i_{sj}} D^{\ell}_{k s} \big) . \label{curvature}
\end{eqnarray}
\end{lemma}

\vspace{0.3cm}

\section{Proof of Theorem \ref{thm}}

Let ${\mathfrak g}$ be  a solvable Lie algebra, equipped with a complex structure $J$ and a  metric $g=\langle ,\rangle$ compatible with $J$. Since Conjecture \ref{conj1} is known for nilmanifolds by \cite{LZ}, we will assume that ${\mathfrak g}$ is not nilpotent. Assume that the commutator ${\mathfrak g}'=[{\mathfrak g},{\mathfrak g}]$ is $J$-invariant. As is well-known, the commutator of a solvable Lie algebra is always nilpotent, so in particular we have  ${\mathfrak g}'\neq {\mathfrak g}$. 

We will use a particular type of unitary frames that will be most convenient for our proofs. First of all, we want our unitary frame $e$  on ${\mathfrak g}$ so that ${\mathfrak g}'$ is spanned by $\{ e_i+\overline{e}_i, \sqrt{-1}(e_i - \overline{e}_i)\}_{1\leq i\leq r}$. Throughout this section, we will always make the following convention on the range of indices:
$$ 1\leq i,j,\cdots \leq r, \ \ \ \ \ \ r\!+\!1\leq \alpha, \beta , \cdots \leq n, \ \ \ \ \ \ 1\leq a,b,\cdots \leq n. $$
By (\ref{CandD2}) and the definition of ${\mathfrak g}'$, we get the following
\begin{equation}  \label{eq:res1}
C^{\alpha}_{ab}=D^{a}_{\alpha b}=0, \ \ \ \ \forall \ r\!+\!1\leq \alpha \leq n, \ \forall \ 1\leq a,b\leq n. 
\end{equation}
Let us restrict the Hermitian structure onto ${\mathfrak g}'$. Then $\{  e_1, \ldots , e_r\}$ becomes a unitary frame for  ${\mathfrak g}'$. Since ${\mathfrak g}'$ is nilpotent, thanks to the beautiful theorem of Salamon \cite[Theorem 1.3]{Salamon}, we know that, by a unitary change of $\{ e_1, \ldots , e_r\}$ if necessary, we may assume that
\begin{equation}  \label{eq:res2}
C^j_{ik}=0  \ \  \mbox{unless} \ j>i \ \mbox{or} \ j>k; \ \ \ \ D^j_{ik}=0 \  \ \mbox{unless} \ i>j, \ \ \ \forall \ 1\leq i,j,k\leq r.
\end{equation}
This will be a crucial property which we will use repeatedly. We refer the readers to \cite{LZ} for a more detailed discussion on this, and the above is simply formula (15) in \cite{LZ}. We will call such a unitary frame $e$ on ${\mathfrak g}$  an {\em admissible frame} from now on.

\begin{lemma} \label{lemmac=0}
Let $({\mathfrak g},J,g)$ be a Lie algebra with Hermitian structure such that ${\mathfrak g}$ is solvable and  $J{\mathfrak g}'={\mathfrak g}'$. Assume that the holomorphic sectional curvature $H$ of the Chern connection of $g$ is a constant $c$. Then $c=0$
\end{lemma}

\begin{proof}
Let $e$ be an admissible frame of ${\mathfrak g}$. Fix any $\alpha$ in $\{ r\!+\!1, \ldots ,n\}$. By the fact $D^{\ast}_{\alpha \ast}=0$ and by formula (\ref{curvature}) in Lemma \ref{lemma2}, we get
\begin{equation} \label{eq:alpha}
 c=R_{\alpha \bar{\alpha} \alpha \bar{\alpha} } = - \sum_{s=1}^n |D^{\alpha }_{s\alpha}|^2 \ \leq \,0. 
 \end{equation}
Similarly, for any $1\leq i\leq r$, we have
\begin{equation}  \label{eq:i}
 c=R_{i \bar{i} i \bar{i} } =  \sum_{s=1}^n |D^s_{ii}|^2 - \sum_{s=1}^r |D^{i}_{si}|^2 ,
 \end{equation}
since the summing index $s$ for the last two terms on the right-hand side of formula (\ref{curvature}) cannot exceed $r$, while $D^i_{is}=0$ by (\ref{eq:res2}). In particular, if we take $i=r$, then the second term on the right-hand side of (\ref{eq:i}) vanishes by (\ref{eq:res2}), so we get 
\begin{equation}  \label{eq:r}
 c=R_{r \bar{r} r \bar{r} } =  \sum_{s=1}^n |D^s_{rr}|^2 \ \geq \, 0. 
 \end{equation}
Comparing (\ref{eq:alpha}) with (\ref{eq:r}), we conclude that $c=0$. 
\end{proof}

\begin{lemma} \label{lemmagoal}
Let $({\mathfrak g},J,g)$ be a Lie algebra with Hermitian structure such that ${\mathfrak g}$ is solvable and  $J{\mathfrak g}'={\mathfrak g}'$. Assume that the holomorphic sectional curvature $H$ of the Chern connection of $g$ is zero. Then under any admissible frame $e$ of ${\mathfrak g}$ it holds
\begin{equation}  \label{eq:goal}
D^{\alpha}_{j\beta}=D^j_{ik}=D^{\alpha}_{ij}=0,  \ \ \ \ \ \ \ \ \forall\ 1\leq i,j,k\leq r, \ \ \ \forall \ r\!+\!1\leq \alpha , \beta \leq n. 
 \end{equation}
\end{lemma}

\begin{proof}
By formula (\ref{eq:alpha}) and the vanishing of the holomorphic sectional curvature, we know that $D^{\alpha}_{j\alpha}=0$ for any $j$ and any $\alpha$. If we make any unitary change on $\{ e_{r\!+\!1},\ldots , e_n\}$, formula (\ref{eq:res1}) and (\ref{eq:res2}) would not be affected and $e$ would still be an admissible frame. So formula (\ref{eq:alpha}) actually gives us $D^{X}_{jX}=0$ for any $j$ and any $X=\sum_{\alpha =r\!+\!1}^n X_{\alpha}e_{\alpha}$. That is, $\sum_{\alpha ,\beta=r\!+\!1}^n \overline{X}_{\alpha} X_{\beta} D^{\alpha}_{j\beta }=0$. This leads to $D^{\alpha}_{j\beta}=0$ for any $j$ and any $\alpha$, $\beta$. 

Analogously, the vanishing of $H$ and (\ref{eq:r}) give us $D^{\ast}_{rr}=0$. Our next goal is to prove the remaining part of (\ref{eq:goal}), namely, 
$D^j_{ik}=D^{\alpha}_{ij}=0$ for any $ i$,\,$j$,\,$k$ and any $ \alpha$. To prove this, let us fix any $1\leq i<k\leq r$ and consider
\begin{eqnarray}
0 & = & 4\widehat{R}_{i\bar{i}k\bar{k}} \ = \ R_{i\bar{i}k\bar{k}} + R_{k\bar{k}i\bar{i}}  + R_{i\bar{k}k\bar{i}} +R_{k\bar{i}i\bar{k}} \nonumber \\
& = & \sum_{s=1}^n \{ \big( |D^s_{ki}|^2 -  |D^k_{si}|^2\big) + \big(|D^s_{ik}|^2 -  |D^i_{sk}|^2\big)  +2\mbox{Re}  \big( D^s_{ki}\overline{D^s_{ik}} - D^i_{si}\overline{D^k_{sk}} - \overline{D^i_{sk}} D^i_{ks} \big) \} \nonumber \\
& = & \sum_{s=1}^n |D^s_{ki}+D^s_{ik}|^2 -  \sum_{s=1}^r \big( |D^k_{si}|^2 + |D^i_{sk}|^2 + 2\mbox{Re} \{ D^i_{si}\overline{D^k_{sk}} + \overline{D^i_{sk}} D^i_{ks} \} \big). \label{eq:iikk}
\end{eqnarray}
In the last equality above, we used the fact that $D^{\ast}_{\alpha \ast}=0$ for any $\alpha >r$ by (\ref{eq:res1}). Let $k=r$ in (\ref{eq:iikk}), and by (\ref{eq:res2}) we get
\begin{equation}  \label{eq:iirr}
 \sum_{s=1}^n |D^s_{ri}+D^s_{ir}|^2 \, = \, \sum_{s=1}^r \big(  |D^i_{sr}|^2 + 2\mbox{Re} \{  \overline{D^i_{sr}} D^i_{rs} \} \big) , \ \ \ \ \ \ \forall \ 1\leq i<r.
 \end{equation}
We claim that the above will lead to the following:
\begin{equation}  \label{eq:claimr}
 D^s_{ri}+D^s_{ir}=D^i_{jr}=D^i_{rj}=0, \ \ \ \ \ \ \forall \ 1\leq i,j\leq r, \ \ \forall \ 1\leq s\leq n.
 \end{equation}
To prove the claim, let us perform induction on $i$. It holds for $i=r$ by (\ref{eq:res2}) and the fact $D^{\ast}_{rr}=0$. Let $i=r\!-\!1$ in (\ref{eq:iirr}). The right-hand side of (\ref{eq:iirr}) vanishes as $s$ has to be $r$ by (\ref{eq:res2}) while $D^{\ast}_{rr}=0$. Therefore we get $D^s_{r,r\!-\!1}+D^s_{r\!-\!1,r} =0$ for any $s$, hence (\ref{eq:claimr}) holds for $i=r\!-\!1$. Next let us take $i=r\!-\!2$ in (\ref{eq:iirr}). On the right hand side, $s$ has to be $r\!-\!1$. But by (\ref{eq:claimr}) for $i=r\!-\!1$ we have $D^{r\!-\!2}_{r\!-\!1,r} + D^{r\!-\!2}_{r,r\!-\!1}=0$. Therefore the right-hand side of (\ref{eq:iirr}) becomes $-| D^{r\!-\!2}_{r\!-\!1,r} |^2$, so we get (\ref{eq:claimr}) for the $i=r\!-\!2$ case. Repeating this process, the Claim (\ref{eq:claimr}) is proved. In particular, we have shown that
\begin{equation*}  \label{eq:claimDr}
 D^i_{jr}=D^i_{rj}=D^r_{ij} =0, \ \ \ \ \ \ \forall \ 1\leq i,j\leq r.
 \end{equation*}
If we look at the $i=r\!-\!1$ case of (\ref{eq:i}), then the second term on the right-hand side vanishes, so we have
\begin{equation}  \label{eq:Dr-1r-1}
 D^s_{r\!-\!1,r\!-\!1}=0, \ \ \ \ \ \ \forall \ 1\leq s\leq n.
 \end{equation}
Now let us take $k=r\!-\!1$ in (\ref{eq:iikk}), and get
\begin{equation}  \label{eq:iir-1r-1}
 \sum_{s=1}^n |D^s_{r\!-\!1,i}+D^s_{i,r\!-\!1}|^2 \, = \, \sum_{s=1}^r \big(  |D^i_{s,r\!-\!1}|^2 + 2\mbox{Re} \{  \overline{D^i_{s,r\!-\!1}} D^i_{r\!-\!1,s} \} \big) , \ \ \ \ \ \ \forall \ 1\leq i<r\!-\!1.
 \end{equation}
It leads to the following by arguing inductively in $i$ in decreasing order as we have seen before:
\begin{equation}  \label{eq:claimr-1}
 D^s_{r\!-\!1,i}+D^s_{i,r\!-\!1}=D^i_{j,r\!-\!1}=D^i_{r\!-\!1,j}=0, \ \ \ \ \ \ \forall \ 1\leq i,j\leq r, \ \ \forall \ 1\leq s\leq n.
 \end{equation}
In particular, we have shown that
$$ D^i_{j,r\!-\!1}=D^i_{r\!-\!1,j}=D^{r\!-\!1}_{ij} =0, \ \ \ \ \ \ \forall \ 1\leq i,j\leq r. $$
In other words, for $D^j_{ik}$ with $1\leq i,j,k\leq r$, it vanishes if any of the index is $r$ or $r\!-\!1$. Repeating this process, we conclude that 
$D^j_{ik}=0$ for all  $1\leq i,j,k\leq r$. This completes the proof of the second equality of (\ref{eq:goal}). To prove the third equality of (\ref{eq:goal}), note that the second term on the right-hand side of (\ref{eq:i}) now vanishes, so it gives us $D^{\alpha}_{ii}=0$. If we perform a unitary change on $\{ e_1, \ldots , e_r\}$, we still have $D^j_{ik}=0$ for all $1\leq i,j,k\leq r$. So what (\ref{eq:i}) gives us is actually $D^{\alpha}_{vv}=0$ for any $v=v_1e_1+\cdots + v_re_r$. In particular we have 
\begin{equation}  \label{eq:Dalphaij}
 D^{\alpha}_{ij}+ D^{\alpha}_{ji} =0 , \ \ \ \ \ \ \ \ \forall \ 1\leq i,j\leq r, \ \ \forall \ r\!+\!1\leq \alpha \leq n.
 \end{equation}
For any fixed $1\leq i\leq r$ and $r\!+\!1\leq \alpha \leq n$, by (\ref{eq:res2}) and the first equality of (\ref{eq:goal}) we have
\begin{eqnarray*}
0 & = & 4\widehat{R}_{i\bar{i}\alpha \bar{\alpha}} \ = \ R_{i\bar{i}\alpha \bar{\alpha}} + R_{\alpha \bar{\alpha}i\bar{i}}  + R_{i\bar{\alpha}\alpha \bar{i}} +R_{\alpha \bar{i}i\bar{\alpha}} \nonumber \\
& = & \sum_{s=1}^r  \{  |D^s_{i\alpha}|^2 -  |D^i_{s\alpha}|^2- |D^{\alpha}_{si}|^2 -  2\mbox{Re}  \big( D^{\alpha}_{si}\overline{D^{\alpha}_{is}}  \big) \} \nonumber \\
& = & \sum_{s=1}^r  \{  |D^s_{i\alpha}|^2 -  |D^i_{s\alpha}|^2 + |D^{\alpha}_{si}|^2  \}.
\end{eqnarray*}
Here in the last equality we used the (\ref{eq:Dalphaij}) for the last term. If we sum $i$ from $1$ to $r$, then the first two terms on the right-hand side of the last line above will cancel each other, and it gives us $D^{\alpha}_{si}=0$ for any $1\leq s,i\leq r$. This completes the proof of the third equality of (\ref{eq:goal}) hence the lemma. 
\end{proof}

Now we are ready to finish the proof of Theorem \ref{thm}.

\begin{proof}[{\bf Proof of Theorem \ref{thm}.}]
Let $({\mathfrak g}, J,g)$ be a solvable Lie algebra equipped with a Hermitian structure. Assume that $J{\mathfrak g}'={\mathfrak g}'$ and the Chern connection of $g$ has constant holomorphic sectional curvature: $H=c$. Then $c=0$ by Lemma \ref{lemmac=0}.  Our goal is  to show that the Chern curvature of $g$ must vanish identically: $R=0$. Let $e$ be an admissible frame of ${\mathfrak g}$ so (\ref{eq:res1}) and (\ref{eq:res2}) hold. By (\ref{eq:goal}) the only possibly non-zero $D$ components are $ D^j_{i\alpha}$, where $ 1\leq i,j\leq r$ and $r\!+\!1 \leq \alpha \leq n$. In particular, $D^{\alpha}_{\ast\ast} = D^{\ast}_{\alpha \ast}=0$. So by (\ref{curvature}) we know that
\begin{equation} \label{eq:alphaat34}
 R_{\ast \bar{\ast} \alpha \bar{\ast}} =  R_{\ast \bar{\ast} \ast \bar{\alpha} }=0, \ \ \ \ \ \ \forall \ r\!+\!1 \leq \alpha \leq n. 
 \end{equation}
Therefore $R_{a\bar{b}c\bar{d}}=0$ if at least three of the indices belong to $\{ r\!+\!1, \ldots , n\}$. Also by (\ref{curvature}) we have
$$ R_{a\bar{b}c\bar{d}} = \sum_{s=1}^n D^s_{ca} \overline{D^s_{db} } - \sum_{s=1}^n D^d_{sa} \overline{D^c_{sb} }, $$
since for the last two terms in the right hand side of (\ref{curvature})  the index $s$ cannot be simultaneously $>r$ and $\leq r$. If exactly two of the four indices belong to $\{ r\!+\!1, \ldots , n\}$, then by (\ref{eq:alphaat34}) we get
$$R_{\alpha \bar{\beta}i\bar{j}}=\widehat{R}_{\alpha \bar{\beta}i\bar{j}}=0. $$
If only one of the four indices belongs to $\{ r\!+\!1, \ldots , n\}$, then it is
$$R_{\alpha \bar{j}k\bar{\ell}}=   \sum_{s=1}^n \big( D^s_{k\alpha } \overline{D^s_{\ell j} } -  D^{\ell}_{s\alpha } \overline{D^k_{sj} }\big).$$
This equals to zero since for each of the two terms on the right-hand side, the second factor is zero by (\ref{eq:goal}). Similarly, when all four indices are in $\{ 1, \ldots , r\}$, $R_{i\bar{j}k\bar{\ell}}=0$ by (\ref{eq:goal}). Therefore, we have shown that $R=0$, completing the proof of Theorem \ref{thm}.
 \end{proof}

\vspace{0.3cm}

\section{Examples}

A special type of Lie algebras satisfying the assumptions of Theorem \ref{thm} are $2$-step solvable Lie algebras of pure type II. In \cite{FSwann, FSwann2}, Freibert and Swann give a systematic study on the Hermitian geometry of $2$-step solvable Lie algebras, namely, solvable ${\mathfrak g}$ which are non-abelian but with ${\mathfrak g}'= [{\mathfrak g}, {\mathfrak g}]$ being abelian. Among other things, they characterized balanced metrics (namely Hermitian metrics satisfying $d(\omega^{n-1})=0$) and pluriclosed metrics (namely Hermitian metrics satisfying $\partial \overline{\partial} \omega =0$) on such Lie algebras, and confirmed Fino-Vezzoni Conjecture for all such Lie algebras that are of {\em pure types}. Here $\omega$ denotes the K\"ahler form of the Hermitian metric, $n$ denotes the complex dimension of the manifold, and Fino-Vezzoni Conjecture (\cite{FV15, FV16}) states  that any compact complex manifold admitting both a balanced metric and a pluriclosed metric must admit a K\"ahler metric. 

Recall that for any  $2$-step solvable ${\mathfrak g}$ equipped with a complex structure $J$, Freibert and Swann introduced the following terminology: ${\mathfrak g}$  is said to be:
\begin{itemize} 
\item Pure type I: if  $J{\mathfrak g}'\cap {\mathfrak g}'=0$,
\item Pure type II: if $J{\mathfrak g}'={\mathfrak g}'$,
\item Pure type III: if $J{\mathfrak g}'+{\mathfrak g}'={\mathfrak g}$.
\end{itemize}
Let $W$ be a $J$-invariant complement subspace of ${\mathfrak g}' + J{\mathfrak g}'$ in ${\mathfrak g}$. Let $V$ be a complement subspace of ${\mathfrak g}'\cap J {\mathfrak g}'$ in ${\mathfrak g}'$, then we have the decomposition of ${\mathfrak g}$ into the direct sum of three $J$-invariant subspaces:
\begin{equation} \label{eq:sumofthree}
 {\mathfrak g} = \big( {\mathfrak g}'\cap J {\mathfrak g}' \big) \oplus \big( V \oplus JV\big) \oplus W, 
 \end{equation}
where the sum of the first two terms is ${\mathfrak g}' + J{\mathfrak g}'$. In terms of the decomposition, ${\mathfrak g}$ if of pure types if and only if at least one of the three summands vanishes. (To be more precise, ${\mathfrak g}$ if of pure type I, II, or III if the first, second, or third summand in (\ref{eq:sumofthree}) vanishes, respectively). Clearly the types are independent of the choice of $V$ and $W$. Note that pure type II does not intersect pure type I or III, while pure type I and pure type III have an overlap. Obviously, $2$-step solvable Lie algebras of pure type II constitute a special case of those in Theorem \ref{thm}, so for any $2$-step solvable Hermitian Lie algebra of pure type II, its holomorphic sectional curvature is constant if and only if it is Chern flat. In the following let us discuss a typical example.

\begin{example}[{\bf 1}]
Consider the Heisenberg Lie algebra ${\mathfrak h}_{\sigma} = U\oplus {\mathbb R}$, where $\sigma$ is a non-degenerate skew-symmetric bilinear form on $U\cong {\mathbb R}^{2r}$, and the Lie bracket is defined by 
$$[(v_1, s_1), (v_2, s_2)] = (0, \sigma (v_1, v_2)). $$ 
Equivalently, it can be described as ${\mathfrak h}_{\sigma}={\mathbb R}\{ X_1, \ldots ,X_r, Y_1, \ldots , Y_r, Z\}$ where the only non-trivial Lie brackets are given by
$$ [X_i, Z] = a_iX_i + b_i Y_i, \ \ \ [Y_i, Z] = -b_iX_i + a_i Y_i, \ \ \ 1\leq i\leq r. $$
Here we assumed that $\lambda_i = -(b_i + \sqrt{-1}a_i) \neq 0$ for each $i$. Consider the Lie algebra ${\mathfrak h}= {\mathfrak h}_{\sigma}\oplus {\mathbb R}$ equipped with an almost complex structure $J$ given by
$$ JW=Z, \ \ \ JZ=-W, \ \ \ JX_i=Y_i, \ \ \ JY_i = - X_i, \ \ \ 1\leq i\leq r, $$
where $W$ is the unit of the ${\mathbb R}$ factor. Then it is easy to see that $J$ is integrable, and the commutator ${\mathfrak h}'=[{\mathfrak h}, {\mathfrak h}]=U$ is abelian and $J$-invariant, thus ${\mathfrak h}$ is $2$-step solvable of pure type II. 
\end{example}

Let $g$ be the metric on ${\mathfrak h}$ compatible with $J$ so that  $\{ X_1, Y_1, \ldots , X_r, Y_r, W, Z\}$ is orthonormal. Write
$$ e_i=\frac{1}{\sqrt{2}} (X_i - \sqrt{-1}Y_i), \ \ (1\leq i\leq r), \ \ \ \ \  e_0= \frac{1}{\sqrt{2}} (W-\sqrt{-1}Z). $$
Then $\{ e_1, \ldots , e_r, e_0\}$ becomes a unitary frame, and we compute 
$$ [e_i, W]=0, \ \ \ [e_i,Z] = (a_i+\sqrt{-1}b_i)e_i = - \overline{\lambda}_i e_i, \ \ \ \forall \ 1\leq i\leq r. $$
Therefore the following are all the non-trivial components of the structure constants under $e$:
$$ C^i_{i0}= -\overline{\lambda}_i, \ \ \ D^i_{i0}=\lambda_i, \ \ \ \mbox{where} \ \ \lambda_i = -(b_i+\sqrt{-1}a_i) \neq 0, \ \ 1\leq i\leq r. $$
Denote by $\{ \varphi_1, \ldots , \varphi_r, \varphi_0\}$ the coframe dual to $e$, then the structure equation becomes
$$ d\varphi_0=0, \ \ \ d\varphi_i = \overline{\lambda}_i \varphi_i \wedge (\varphi_0 - \overline{\varphi}_0), \ \ \ 1\leq i\leq r. $$
By a straight-forward computation, one can verify that the above $({\mathfrak h}, J,g)$ is indeed Chern flat. 

\vspace{0.2cm}

Next, let us examine those $({\mathfrak g}, J,g)$ in Theorem \ref{thm} that are Chern flat. Let $e$ be an admissible frame of ${\mathfrak g}$. From the discussion in the previous section, we know that $C^{\alpha}_{\ast\ast}=D^{\alpha}_{\ast\ast} = D^{\ast}_{\alpha \ast}=0$ for any $r<\alpha \leq n$. For any $1\leq i,j\leq r$ and any $r<\alpha ,\beta\leq n$, by (\ref{curvature}) we have
$$ R_{\alpha \bar{\beta}i\bar{j}} = \sum_{s=1}^r \big( D^s_{i\alpha} \overline{D^s_{j\beta} } - D^j_{s\alpha} \overline{D^i_{s\beta} } \big) . $$
Therefore the vanishing of $R$ implies that $[D_{\alpha}, D^{\ast}_{\beta}]=0$, where $D_{\alpha} = (D^j_{i\alpha})$ is the $r\times r$ matrix, and $A^{\ast}$ denotes the transpose conjugation of $A$. By the middle line in (\ref{Bianchi}), we get $[D_{\alpha}, D_{\beta}]=0$ for any $r<\alpha, \beta \leq n$. So these $D_{\alpha}$ are commuting normal matrices, thus can be simultaneously diagonalized, namely, by a unitary change of $\{ e_1, \ldots , e_r\}$ if necessary, we may assume that
\begin{equation}  \label{eq:Y}
 D^j_{i\alpha }=Y_{i\alpha} \delta_{ij}, \ \ \ \ \ \ \ \forall \ 1\leq i,j\leq r, \ \ \forall \ r<\alpha \leq n. 
 \end{equation}
Note that due to the frame change here, (\ref{eq:res2}) may no longer hold. The third line of (\ref{Bianchi}) gives us
\begin{equation}  \label{eq:CY}
 \overline{Y_{i\alpha}} \,C^i_{\beta\gamma} = (\overline{Y_{i\alpha}} - \overline{Y_{j\alpha}})\, C^j_{i\beta} = (\overline{Y_{j\alpha}} - \overline{Y_{i\alpha}} - \overline{Y_{k\alpha}}) \,C^j_{ik} =0,  
 \end{equation}
for any $ 1\leq i,j,k\leq r$ and any $ r<\alpha, \beta, \gamma \leq n$. Let us rearrange the order of $\{ e_1, \ldots , e_r\}$ so that equal rows of $Y$ are put into the same blocks, namely, write $\{ 1, \ldots , r\}$ as disjoint union $I_0\cup I_1 \cup \cdots \cup I_{\ell}$, where 
\begin{eqnarray*}
&&  I_0=\{ 1, \ldots , r_0\},\  \  \mbox{with}  \ 0\leq r_0\leq r; \ \  \mbox{and} \ I_1=\{ r_0\!+\!1, \ldots , r_1\};\\
&&  I_2=\{ r_1\!+\!1, \ldots , r_2\}; \ \ldots \ldots ; \ I_{\ell}=\{r_{\ell\!-\!1}\!+\!1, \ldots , r_{\ell}\}, \ \mbox{where} \ r_{\ell}=r,
\end{eqnarray*}
so that for each $j$,  if $i,k\in I_j$, then $Y_{i\alpha}=Y_{k\alpha}$ for any $\alpha$ and vice versa. Also, $I_0$ corresponds to the zero rows in $Y$. Under this order, (\ref{eq:CY}) implies that for any $r<\alpha, \beta \leq n$ we have
\begin{equation}  \label{eq:CY2}
 C^i_{\alpha \beta}=0 \ \mbox{unless}\ i\in I_0; \ \mbox{while each} \ r\times r \ \mbox{matrix} \ C_{\alpha} =(C^j_{i\alpha})\ \mbox{is block-diagonal.} 
 \end{equation}
Write ${\mathfrak a}={\mathfrak g}/{\mathfrak g}'$ for the abelian Lie algebra of dimension $2(n-r)$, then ${\mathfrak g}$ is the extension of the nilpotent complex Lie algebra ${\mathfrak g}'$ by the abelian Lie algebra ${\mathfrak a}$
$$ 0 \rightarrow {\mathfrak g}' \rightarrow {\mathfrak g} \rightarrow {\mathfrak a}  \rightarrow  0, $$
satisfying (\ref{eq:Y}), (\ref{eq:CY}), and the first line of (\ref{Bianchi}).  

\begin{example}[{\bf 2}]
When $r_0=r$, or equivalently $Y=0$, then ${\mathfrak g}$ is a (solvable) complex Lie algebra. In this case it is well-known that any metric compatiblle with the complex structure is Chern flat.
\end{example}

\begin{example}[{\bf 3}]
When $C^j_{ik}=0$ for all $1\leq i,j,k\leq r$, or equivalently when ${\mathfrak g}'$ is abelian, then ${\mathfrak g}$ is $2$-step solvable of pure type II. In this case the first line of (\ref{Bianchi}) means (\ref{eq:CY2}) and 
$$[C_{\alpha}, C_{\beta}]=0, \ \ \ \ \,^t\!C_{\alpha} w_{\beta\gamma} + \,^t\!C_{\beta} w_{\gamma\alpha}+\,^t\!C_{\gamma} w_{\alpha\beta} =0, \ \ \ \ \ \forall \ r<\alpha, \beta , \gamma \leq n, $$ 
where $w_{\alpha\beta}$ is the column vector in ${\mathbb C}^r$ with entries $C^i_{\alpha\beta}$. 
\end{example}

In the following, let us give an explicit example of the above kind, when the only possibly non-zero components of $C$ are $C^i_{i\alpha}$:

\begin{example}[{\bf 3}]
Fix any integers $1\leq r<n$. Let $V\cong {\mathbb R}^{2n}$ be a vector space and $\{ \varepsilon_1, \ldots , \varepsilon_{2n}\}$ be a basis of $V$. Let $g=\langle ,\rangle$ be the metric on $V$ so that $\varepsilon$ is orthonormal, and $J$ the almost complex structure on $V$ so that $J\varepsilon_i = \varepsilon_{n+i}=\varepsilon_{i^{\ast}}$ for $1\leq i\leq n$. For any $1\leq i\leq r$ and any $r<\alpha \leq n$, let $x_{i\alpha}$, $y_{i\alpha}$, $u_{i\alpha}$, $v_{i\alpha}$ be constants and define Lie brackets 
\begin{eqnarray*}
&&  [\varepsilon_i, \varepsilon_{\alpha}] = x_{i\alpha} \varepsilon_i + y_{i\alpha} \varepsilon_{i^{\ast}},  \ \ \ \ \ \ [\varepsilon_{i^{\ast}}, \varepsilon_{\alpha}] = -y_{i\alpha} \varepsilon_i + x_{i\alpha} \varepsilon_{i^{\ast}}, \\
&& [\varepsilon_i, \varepsilon_{\alpha^{\ast}}] = v_{i\alpha} \varepsilon_i -u_{i\alpha} \varepsilon_{i^{\ast}},  \ \ \ \ \ [\varepsilon_{i^{\ast}}, \varepsilon_{\alpha^{\ast}}] = -u_{i\alpha} \varepsilon_i + v_{i\alpha} \varepsilon_{i^{\ast}},
\end{eqnarray*}
while all other Lie brackets are zero. Then it is easy to verify that $V$ becomes a Lie algebra ${\mathfrak g}$ and $J$ is integrable. It is also easy to see that ${\mathfrak g}$ is $2$-step solvable of pure type II and $g$ is Chern flat.
\end{example}

In fact, if we let $e_i=\frac{1}{\sqrt{2}}(\varepsilon_i - \sqrt{-1}\varepsilon_{i^{\ast}})$ for $1\leq i\leq n$, then in terms of the unitary frame $e$ the only possibly non-zero Lie brackets are
$$ [e_i, e_{\alpha}] = Z_{i\alpha} e_i, \ \ \ \ [\overline{e}_i, e_{\alpha}] = Y_{i\alpha} \overline{e}_i, \ \ \ \ \ \ 1\leq i\leq r, \ r<\alpha \leq n, $$
where $Y=\frac{1}{\sqrt{2}}\{ (x+u)-\sqrt{-1}(y+v)\}$ and $Z=\frac{1}{\sqrt{2}}\{ (x-u)+\sqrt{-1}(y-v)\}$ for each $i$ and $\alpha$. As long as for each $i$, there will be some $\alpha$ such that $(Y_{i\alpha}, Z_{i\alpha})\neq (0,0)$, then the commutator ${\mathfrak g}'$ will be the entire $\mbox{span}\{ \varepsilon_1, \ldots , \varepsilon_r, \varepsilon_{1^{\ast}}, \ldots , \varepsilon_{r^{\ast}}\}$. 

Finally, we remark that for solvable Lie algebras with complex commutators that are Chern flat, the majority of them are not $2$-step solvable. For instance, when there are some indices $1\leq i,j,k\leq r$ so that $C^j_{ik}\neq 0$, then the Lie algebra is not $2$-step solvable. There are plenty of such Hermitian Lie algebras, and here let us give an explicit example of such kind:

\begin{example}[{\bf 4}]
Consider the $6$-dimensional Lie algebra ${\mathfrak g}$ equipped with a Hermitian structure $(J,g)$, where under some unitary frame $e$ the only non-trivial Lie brackets are
$$ [e_1, e_2]=e_1;\ \ [e_i,e_3]=Z_ie_i,\ \ [\overline{e}_i,e_3]=Y_i\overline{e}_i, \ i=1,2. $$
Assume that $(Y_1,Z_1)\neq (0,0)$ and $(Y_2,Z_2)\neq (0,0)$. Then ${\mathfrak g}$ is solvable, $J{\mathfrak g}'={\mathfrak g}'$, and $g$ is Chern flat. It is clear that ${\mathfrak g}'$ is not abelian, so ${\mathfrak g}$ is not $2$-step solvable. It is $3$-step solvable due to the low dimension here. 
\end{example}
 
It terms of the unitary coframe $\varphi$  dual to $e$, the above Lie bracket definition can be equivalently described by the (first) structure equation
\begin{equation*}
  d\varphi_1 = -\varphi_1\wedge \varphi_2 -\varphi_1\wedge (Z_1\varphi_3+\overline{Y}_1\overline{\varphi}_3), \ \   d\varphi_2 =  -\varphi_1\wedge (Z_2\varphi_3+\overline{Y}_2\overline{\varphi}_3), \ \   d\varphi_3 = 0. 
\end{equation*}
Since under any unitary frame $e$, the connection matrix $\theta$ for the Chern connection is given by 
$$ \theta_{ij}= \sum_{k=1}^n  \big(  D^j_{ik} \varphi_k - \overline{D^i_{jk}} \,\overline{\varphi}_k \big), $$ 
we know that in the case of the above example we have 
$$ \theta = \left[ \begin{array}{ccc} \theta_{11} & 0 & 0 \\     0 & \theta_{22} & 0 \\ 0 & 0 & 0 \end{array}  \right] , \ \ \ \theta_{ii} = Y_i\varphi_3-\overline{Y}_i\overline{\varphi}_3, \ \ i=1,2. $$
Therefore the curvature matrix under $e$ is $\Theta = d\theta -\theta \theta = d\theta = 0$, hence the metric is Chern flat.

\vspace{0.3cm}

\vs

\noindent\textbf{Acknowledgments.} The second named author would like to thank Haojie Chen, Xiaolan Nie, Sheng Rao, Kai Tang, Bo Yang, Xiaokui Yang, and Quanting Zhao for their interests and/or helpful discussions. We are very grateful to the anonymous referee for a number of highly valuable suggestions and corrections, which clarified things and improved the readability of the article. The authors are ranked in alphabetic order of their names and should be treated equally.

\vs

\end{document}